\documentclass[11pt,reqno,twoside]{amsart}
\usepackage{graphicx}
\usepackage{amssymb}
\usepackage{epstopdf}
\usepackage[asymmetric,top=3.5cm,bottom=4.3cm,left=3.1cm,right=3.1cm]{geometry}
\usepackage{booktabs} 
\usepackage{array} 
\usepackage{paralist} 
\usepackage{verbatim} 
\usepackage{subfig} 
\usepackage{tabularx}
\usepackage{amsmath,amsfonts,amsthm,mathrsfs,amssymb,cite}
\usepackage[usenames]{color}
\usepackage{bm}

\DeclareMathOperator{\spn}{span}

\newtheorem{thm}{Theorem}[section]

\newtheorem{lem}{Lemma}[section]
\newtheorem{prop}{Proposition}[section]
\theoremstyle{definition}

\theoremstyle{remark}

\newtheorem{rem}{Remark}[section]
\numberwithin{equation}{section}

\newcommand{\bu}{\mathbf{u}}

\newcommand{\bw}{\mathbf{w}}

\newcommand{\bvarphi}{\bm{\varphi}}

\newcommand{\bnu}{\bm{\nu}}
\newcommand{\bGa}{\mathbf{\Gamma}}
\newcommand{\bx}{\mathbf{x}}
\newcommand{\by}{\mathbf{y}}
\newcommand{\rmi}{\mathrm{i}}
\newcommand{\bS}{\mathbf{S}}
\newcommand{\bK}{\mathbf{K}}
\newcommand{\bI}{\mathbf{I}}

\newcommand{\Acal}{\mathcal{A}}

\newcommand{\Lcal}{\mathcal{L}}

\newcommand{\beq}{\begin{equation}}
\newcommand{\eeq}{\end{equation}}

\newcommand{\mm}[1]{{\color{black}{#1}}}
\newcommand{\rr}[1]{{\color{black}{#1}}}

\title[Minnaert resonances for bubbles in soft elastic materials]{Minnaert resonances for bubbles \\ in soft elastic materials}

\author{Hongjie Li}
\address{Department of Mathematics, The Chinese University of Hong Kong, Shatin, Hong Kong SAR, China.}
\email{hjli@math.cuhk.edu.hk}

\author{Hongyu Liu}
\address{Department of Mathematics, City University of Hong Kong, Kowloon, Hong Kong SAR, China.}
\email{hongyliu@cityu.edu.hk; hongyu.liuip@gmail.com}

\author{Jun Zou}
\address{Department of Mathematics, The Chinese University of Hong Kong, Shatin, Hong Kong SAR, China.
}
\email{zou@math.cuhk.edu.hk}

\begin{document}

\maketitle

\begin{abstract}

Minnaert resonance is a widely known acoustic phenomenon and it has many important applications, in particular in the effective realisation of acoustic metamaterials using bubbly media in recent years. In this paper, motivated by the Minnaert resonance in acoustics, we consider the low-frequency resonance for acoustic bubbles embedded in soft elastic materials. This is a hybrid physical process that couples the acoustic and elastic wave propagations. By delicately and subtly balancing the acoustic and elastic parameters as well as the geometry of the bubble, we show that Minnaert resonance can occur (at least approximately) for rather general constructions. Our study poses a great potential for the effective realisation of negative elastic materials by using bubbly elastic media.  

\medskip

\medskip

\noindent{\bf Keywords:}~~Minnaert resonance, bubbly elastic medium, hybrid Neumann-Poincar\'e operator, spectral, negative elastic materials. 

\noindent{\bf 2010 Mathematics Subject Classification:}~~35R30, 35B30, 35Q60, 47G40

\end{abstract}

\section{Introduction}

The oscillation of bubbles in media is a classical problem. In particular, when bubbles are immersed in liquids, even a very small volume fraction of bubbles can have a significant influence on the effective velocity of waves in liquids \cite{CP, HT}. This is due to the high oscillation of the bubbles caused by the high contrast in density between the bubbles and the surrounding liquid \cite{AF1, GP}. In fact, at a particular low frequency known as the \emph{Minnaert resonant frequency}, the bubbles can be treated as acoustic resonators \cite{M}. The exceptional acoustic properties mentioned above can have many important applications and in particular can be used to design new materials, such as phononic crystals.
In addition to many experimental progresses, the bubbly acoustic materials have been systematically and comprehensively investigated recently in the mathematical literature. Furthermore, based on the mathematical theory developed, novel applications were also proposed, especially for the effective realisation of acoustic metamaterials. For the case that a single bubble is immersed in liquids, the authors in \cite{AF1} provided a rigorous treatment of the Minnaert resonance and the monopole approximation. Later, they investigated the acoustic scattering by a large number of bubbles in liquids at frequencies near the Minnaert resonant frequency in \cite{AF4}. Thus by designing bubble metascreens, the superabsorption effect can be achieved \cite{AF2}. Around the Minnaert resonant frequency, an effective medium theory was derived in \cite{AZ}. Moreover, the opening of a sub-wavelength phononic bandgap was demonstrated by considering a periodic arrangement of bubbles and exploiting the corresponding Minnaert resonance in \cite{AF3}.

Nevertheless, as pointed out in \cite{LB, T}, the practical constructions of acoustic bublly designs are very challenging. The major difficulty arises from making bubbles to have a uniform size and letting them remain inside the liquids. In order to overcome those challenges, substituting the host medium from liquids to soft elastic materials (the shear modulus is small) becomes a more practical scheme. In fact, the oscillation of a spherical cavity in an elastic material was investigated many years ago \cite{MP}. When a spherical bubble is immersed in a soft elastic material, it was shown in \cite{CTL} that there also exists a certain low-frequency resonance. Using such resonant properties, the bubbly elastic structures have been used to the experimental design of new materials for strikingly new applications. For examples, bubble phononic crystals were designed in \cite{LB}, superabsorption of acoustic waves with bubble metascreens was achieved in \cite{LS} and reducing underwater sound transmission was shown in \cite{CTLC} by microfabricating cavities into silicone rubber (a soft elastic material).

Motivated by the aforementioned physical and mathematical studies, we consider in this paper the low-frequency resonance for the case where a bubble is embedded in a soft elastic medium. We aim to derive a systematic and comprehensive mathematical understanding of the resonance phenomena caused by the acoustic and elastic interactions. 
It turns out that the mathematical investigation on the resonance associated with the elastic bubbly media is more challenging than that for the acoustic bubbly media. Indeed, we note that, first, the wave scattering from an elastic bubbly medium is a hybrid physical process which couples the acoustic wave propagation inside the bubble and the elastic wave propagation outside the bubble. Second, since the shear modulus in the elastic material is non-zero, thus the resonance heavily depends on the geometry of the bubble \cite{CTL}. This property is in a sharp contrast to the case of bubbles in liquids which features weak shape dependence \cite{AF1}. Therefore one can not expect an explicit expression of the resonant frequencies (unless the geometry of the bubble is simple, say, a radial one) as the case for bubbles in liquids that was derived in \cite{AF1}. Third, the bubble-liquid resonance only depends on the high contrast of the density between the bubble and the liquid. However, for the bubble-elastic material resonance, in addition to the high contrast of the density, the high contrast of the shear modulus and the compression modulus is required; see also \cite{SM} for a related discussion.

According to our discussion above, it is clear that the bubble-elastic resonance is of a different physical nature to the bubble-liquid resonance. Nevertheless, in order to reveal its origin of motivation as well as for the terminological convenience, we still call it the Minnaert resonance in the present paper. In order to derive the resonance results, following the spirit of the mathematical treatment in \cite{AF1}, we rely on the layer-potential techniques, which boil down our study to the asymptotic and spectral analysis of the layer-potential operators involved for the coupled PDE systems. By delicately and subtly balancing the acoustic and elastic parameters as well as the geometry of the bubble, we show that Minnaert resonance can (at least approximately) occur for a rather general construction in the three-dimensional case. In the two-dimensional case, due to the technical constraint, we can only deal with the case that the bubble is in the radial geometry. Moreover, as also mentioned earlier, we only consider the case with a single bubble embedded in a soft elastic material. We shall study the other case, e.g., the scattering from multiple bubbles, in our forthcoming work. It is emphasized that similar to the bubble-liquid case, our study poses a great potential for the effective realisation of negative elastic materials by using bubbly elastic media, which we shall also investigate in our near-future study. 

The rest of the paper is organized as follows. In Section 2, we present the general mathematical formulation of our study, especially the acoustic-elastic wave scattering from an bubble-elastic structure and its integral reformulation. In Section 3, we discuss about the general requirements on the medium configuration and also derive some auxiliary results for the subsequent use. Sections 4 and 5 are respectively devoted to the Minnaert resonances in three and two dimensions. Our study is concluded in Section 6 with some related remarks. 

\section{Mathematical setup}

 In this section, we present the general mathematical formulation of our study. 
{Consider an air bubble $D$ in our study, and $D$ is assumed to be a bounded domain in $\mathbb{R}^N$ ($N=2,3$)}, 
with a $C^2$-regular boundary $\partial D$.  Let $\rho_b\in\mathbb{R}_+$ and $\kappa\in\mathbb{R}_+$ signify the density and the bulk modulus of the air inside the bubble, respectively. Assume that the background $\mathbb{R}^N\backslash\overline{D}$ is occupied by a regular and isotropic elastic material parameterized by the Lam\'e constants $(\tilde{\lambda},\tilde{\mu})$ satisfying the following strong convexity conditions 
 \begin{equation}\label{eq:con}
  \mathrm{i)}~~\tilde{\mu}>0\qquad\mbox{and}\qquad \mathrm{ii)}~~N\tilde{\lambda}+2\tilde{\mu}>0.
 \end{equation}
The density of the background material is set to be $\rho_e\in\mathbb{R}_+$. Let $\bu^i$ be an incident elastic wave, which is an entire solution to  $\mathcal{L}_{\tilde{\lambda}, \tilde{\mu}}\bu + \omega^2\rho_e\bu =0$ in $\mathbb{R}^N$. Here, $\omega\in\mathbb{R}_+$ denotes the frequency of the elastic wave. The acoustic-elastic wave interaction is described by the following coupled PDE system (cf. \cite{OS}):
\begin{equation}\label{eq:mo}
  \left\{
    \begin{array}{ll}
      \mathcal{L}_{\tilde{\lambda}, \tilde{\mu}}\bu(\bx) + \omega^2\rho_e\bu(\bx) =0 & \bx\in \mathbb{R}^N\backslash \overline{D},\medskip\\
        \triangle u(\bx) +\tilde{k}^2 u(\bx) =0    & \bx\in D,\medskip \\
   \mathbf{u}(\bx)\cdot\bnu - \frac{1}{ \rho_b \omega^2} \nabla u(\bx)\cdot \bnu=0      & \bx\in\partial D,\medskip \\
      \partial_{\tilde{\bnu}}\mathbf{u}(\bx) +  u(\bx)\bnu=0 & \bx\in\partial D,\medskip\\
     \bu(\bx)-\bu^i(\bx)  \qquad \quad \mbox{satisfies the radiation condition},
    \end{array}
  \right.
\end{equation} 
where $\bu$ is the total elastic wave field outside the domain $D$, $u$ is the pressure inside the domain $D$, $\omega\in\mathbb{R}_+$ is the angular frequency and $\tilde{k}=\omega/c_b$ with $c_b=\sqrt{\kappa/\rho_b}$ signifying the velocity of the wave in $D$. In \eqref{eq:mo}, the Lam\'e operator $ \mathcal{L}_{\tilde{\lambda}, \tilde{\mu}}$ and the co-normal derivative $\partial_{\tilde{\bnu}}$, associated with the parameters $(\tilde{\lambda}, \tilde{\mu})$, are respectively defined by 
\begin{equation}\label{op:lame}
 \Lcal_{\tilde{\lambda}, \tilde{\mu}}\bw:=\tilde{\mu} \triangle\bw + (\tilde{\lambda}+ \tilde{\mu})\nabla\nabla\cdot\bw,
\end{equation}
and
\begin{equation}\label{eq:trac}
\partial_{\tilde{\bnu}}\bw=\tilde{\lambda}(\nabla\cdot \bw)\bnu + 2\tilde{\mu}(\nabla^s\bw) \bnu.
\end{equation}
Here $\bnu$ represents the outward unit normal vector to $\partial D$ and the operator $\nabla^s$ is the symmetric gradient
 \begin{equation}\label{eq:sg1}
 \nabla^s\mathbf{w}:=\frac{1}{2}\left(\nabla\mathbf{w}+\nabla\mathbf{w}^t \right),
 \end{equation}
 with $\nabla\bw$ denoting the matrix $(\partial_j w_i)_{i,j=1}^N$ and the superscript $t$ signifying the matrix transpose. In \eqref{eq:mo}, the third condition denotes the continuity of the normal component of the displacement on the boundary $\partial D$ and the fourth condition is the continuity of the stress across $\partial D$. Moreover, the radiation condition in \eqref{eq:mo} designates the following condition as $|\mathbf{x}|\rightarrow+\infty$ (cf.\cite{LLL1}),
\begin{equation}\label{eq:radi}
\begin{split}
(\nabla\times\nabla\times ( \bu-\bu^i))(\bx)\times\frac{\bx}{|\bx|}-\mathrm{i}\tilde{k}_s\nabla\times ( \bu-\bu^i)(\bx)=&\mathcal{O}(|\bx|^{-2}),\\
\frac{\bx}{|\bx|}\cdot      \left( \nabla(\nabla\cdot ( \bu-\bu^i))  \right)        (\bx)-\mathrm{i}\tilde{k}_p\nabla ( \bu-\bu^i)(\bx)=&\mathcal{O}(|\bx|^{-2}),
\end{split}
\end{equation}
where $\rmi=\sqrt{-1}$,
\begin{equation}\label{pa:ksp}
 \tilde{k}_s=\frac{\omega}{\tilde{c}_s}=\frac{\omega}{\sqrt{\tilde{\mu}/\rho_e}} \quad \mbox{and} \quad \tilde{k}_p=\frac{\omega}{\tilde{c}_p}=\frac{\omega}{\sqrt{(\tilde{\lambda}+2\tilde{\mu})/\rho_e}},
\end{equation}
with $\tilde{\lambda}$ and $\tilde{\mu}$ defined in \eqref{eq:con}.

Next we apply the potential theory to derive the integral representation of the solution to the system \eqref{eq:mo} and give the definition of the resonance. First, we introduce the potential operators for the Helmholtz system and Lam\'e system. Let $G^k(\bx)$ be the fundamental solution of the operator $\triangle+k^2$, namely
\begin{equation}\label{eq:fu_he}
 G^{k}(\bx)=
\left\{
   \begin{array}{ll}
    \displaystyle{ -\frac{\rmi}{4} H_0^{(1)}(k|\bx|)}, & N=2 ,\medskip \\
     \displaystyle{-\frac{e^{\rmi k|\bx|}}{4\pi |\bx|}}, & N=3 ,
   \end{array}
 \right.
\end{equation}
where $H_0^{(1)}$ is the zeroth-order Hankel function of the first kind. The single layer potential associated with the Helmholtz system is defined {for $\varphi(\bx)\in L^2(\partial D)$ by}
\begin{equation}\label{eq:s_h}
 S_{\partial D}^{k}[\varphi](\bx)=\int_{\partial D}G^k(\bx-\by)\varphi(\by)ds(\by) \quad \bx\in\mathbb{R}^N\,.
\end{equation}
Then the conormal derivative of the single layer potential enjoys the jump formula
\begin{equation}\label{eq:ju_he}
\nabla S_{\partial D}^{k}[\varphi] \cdot \bnu|_{\pm}(\bx)=\left(\pm\frac{1}{2}I +K_{\partial D}^{k,*} \right)[\varphi](\bx) \quad \bx\in\partial D,
\end{equation}
where
\[
 K_{\partial D}^{k,*}[\varphi](\bx)= \mbox{p.v.}\int_{\partial D} \nabla_{\bx} G^k(\bx-\by)\cdot \bnu_{\bx}\varphi(\by)ds(\by) \quad \bx\in\partial D,
\]
which is also known as the Neumann-Poincar\'e operator associated with Helmholtz system. Here and also in what follows, $\mbox{p.v.}$ {stands} for the Cauchy principal value. Moreover, {we introduce the following $L^2$-adjoint of the operator $K_{\partial D}^{k,*}$:}
\[
K_{\partial D}^k [\varphi](\bx)= \mbox{p.v.}\int_{\partial D} \nabla_{\by} G^k(\bx-\by)\cdot \bnu_{\by}\varphi(\by)ds(\by) \quad \bx\in\partial D.
\]
In what follows, we denote {$S_{\partial D}^{k}, K_{\partial D}^{k,*}$ and 
$K_{\partial D}^k $} by $S_{\partial D,0}$, $K_{\partial D,0}^{*}$ and $K_{\partial D,0}$ for $k=0$. We would like to point out that, the operators $K_{\partial D,0}^{*}$ and $K_{\partial D,0}$ have the following expressions in three dimensions:
\begin{equation}\label{eq:dekk}
\begin{split}
& K_{\partial D,0}^{*}[\varphi](\bx)= \int_{\partial D} \frac{\langle\bx-\by,\bnu_\bx\rangle}{4\pi|\bx-\by|^3} \varphi(\by)ds(\by) \quad \bx\in\partial D, \\
& K_{\partial D,0} [\varphi](\bx)= \int_{\partial D} \frac{\langle\by-\bx,\bnu_\by\rangle}{4\pi|\bx-\by|^3} \varphi(\by)ds(\by) \quad \bx\in\partial D.
\end{split}
\end{equation}
We refer to \cite{CK,Ned} for the mapping properties of the operators introduced above. 

Next we introduce the potential operators for the Lam\'e system. The fundamental solution $\bGa^{\omega}=(\Gamma^{\omega}_{i,j})_{i,j=1}^N$ for the operator $\Lcal_{\tilde{\lambda},\tilde{\mu}}+\rho_e\omega^2$ can be decomposed into shear and pressure components (cf.\cite{AB:bk}):
\begin{equation}\label{eq:ef}
 \bGa^{\omega}=\bGa^{\omega}_s + \bGa^{\omega}_p,
\end{equation}
where 
\[
 \bGa^{\omega}_p=-\frac{1}{\rho_e \omega^2} \nabla \nabla G^{\tilde{k}_p} \quad \mbox{and} \quad  \bGa^{\omega}_s =\frac{1}{\rho_e \omega^2}(\tilde{k}_s^2 \mathbf{I} + \nabla \nabla) G^{\tilde{k}_s},
\]
with $\mathbf{I}$ denoting the $N\times N$ identity matrix, $G^k$ given in \eqref{eq:fu_he} and $\tilde{k}_s$ as well as $\tilde{k}_p$ defined in \eqref{pa:ksp}.
 The single layer potential operator associated with the fundamental solution $\bGa^{\omega}$ is defined by 
\begin{equation}\label{eq:single}
 \bS_{\partial D}^{\omega}[\bvarphi](\bx)=\int_{\partial D} \bGa^{\omega}(\bx-\by)\bvarphi(\by)ds(\by), \quad \bx\in\mathbb{R}^N,
\end{equation}
for $\bvarphi\in L^2(\partial D)^N$. On the boundary $\partial D$, the conormal derivative of the single layer potential satisfies the following jump formula
\begin{equation}\label{eq:jump}
 \frac{\partial \bS_{\partial D}^{\omega}[\bvarphi]}{\partial \bnu}|_{\pm}(\bx)=\left( \pm\frac{1}{2}\bI + \left(\bK_{\partial D}^{\omega}\right)^*  \right)[\bvarphi](\bx) \quad \bx\in\partial D,
\end{equation}
where
\begin{equation}\label{eq:lanp}
 \bK_{\partial D}^{\omega,*}[\bvarphi](\bx)=\mbox{p.v.} \int_{\partial D} \frac{\partial \bGa^{\omega}}{\partial \bnu(\bx)}(\bx-\by)\bvarphi(\by)ds(\by),
\end{equation}
with the subscript $\pm$ indicating the limits from outside and inside $D$, respectively. The operator $\bK_{\partial D}^{\omega,*}$ is called Neumann-Poincar\'e (N-P) operator of the Lam\'e system. In our subsequent analysis, we also need the following single layer potential operators associated with the p-wave (pressure wave) and s-wave (shear wave), respectively, 
\begin{equation}\label{eq:sps}
 \bS_{\partial D}^{\omega,i}[\bvarphi](\bx)=\int_{\partial D} \bGa^{\omega}_i(\bx-\by)\bvarphi(\by)ds(\by), \quad \bx\in \mathbb{R}^N\backslash \overline{D},
\end{equation}
where $\bvarphi(\by)\in L^2(\partial D)^N $ and the kernel functions $\bGa^{\omega}_i$, $i=p,s$ are defined in \eqref{eq:ef}. We refer to \cite{AB:bk} for the mapping properties of the operators introduced above. 
 
 With the help of the potential operators introduced above, the solution to \eqref{eq:mo} can be represented by the following integral ansatz:
\begin{equation}\label{eq:solN}
  \bu=
 \left\{
   \begin{array}{ll}
     S_{\partial D}^{\tilde{k}}[\varphi_b](\bx), & \bx\in D, \medskip\\
     \bS_{\partial D}^{\omega}[\bvarphi_e](\bx) +\bu^i, &  \bx\in \mathbb{R}^N\backslash \overline{D},
   \end{array}
 \right.
\end{equation}
for some density functions $\varphi_b\in L^2(\partial D)$ and $\bvarphi_e \in L^2(\partial D)^N$. By matching the transmission conditions on the boundary $\partial D$, along with the help of the jump formulas \eqref{eq:ju_he} and \eqref{eq:jump}, 
{it can be} verified by some straightforward calculations that the density functions  $\varphi_b, \bvarphi_e$ satisfy the following system of boundary integral equations:
\begin{equation}\label{eq:defAN}
\tilde{\Acal}(\omega,\delta) [\Phi](\bx)=F(\bx), \quad \bx\in\partial D,
\end{equation}
where
\[
\tilde{\Acal}(\omega,\delta)=  \left(
    \begin{array}{cc}
        \frac{1}{\rho_b \omega^2}\left(-\frac{I}{2} + K^{\tilde{k},*}_{\partial D}\right) &  - \bnu\cdot {\bS}_{\partial D}^{\omega}\medskip \\
     \bnu S^{\tilde{k}}_{\partial D} & \frac{I}{2} + {\bK}^{\omega,*}_{\partial D}\\
    \end{array}
  \right),
  \;
  \Phi= \left(
    \begin{array}{c}
      \varphi_b \\
     \bvarphi_e \\
    \end{array}
  \right)
   \; \mbox{and} \;
  F= \left(
    \begin{array}{c}
    \bnu\cdot \bu^i \\
     -\partial_{\bnu}  \bu^i \\
    \end{array}
  \right).
\]
Then the Minnaert resonance of the system \eqref{eq:mo} is defined for all $\omega\in\mathbb{C}$ such that the following equation holds:
\begin{equation}\label{eq:conre1}
 \tilde{\Acal}(k,\delta)[\Phi](\bx)=0,
\end{equation} 
for a nontrivial solution $\Phi\in\mathcal{H}$. {For notational convenience, 
we shall write} $\mathcal{H}:=L^2(\partial D)\times L^2(\partial D)^N$. Moreover, in our subsequent {study of} 
the Minnaert resonance, we may weaken the condition \eqref{eq:conre1} by finding a solution $\Phi\in\mathcal{H}$ with $\| \Phi\|_{\mathcal{H}}=1$ such that for $k\ll1$,
\begin{equation}\label{eq:conrew1}
 \|\tilde{\Acal}(\omega,\delta)[\Phi](\bx)\|_{\mathcal{H}}\ll1.
\end{equation}
If the condition \eqref{eq:conrew1} is fulfilled, we say that the weak (Minnaert) resonance occurs for the system \eqref{eq:mo}. In contrast to the weak Minnaert resonance, when the condition \eqref{eq:conre1} is fulfilled, we say that the strong (Minnaert) resonance occurs.

Finally, we give a remark on the definition of the Minnaert resonance introduced above. In fact, the definition of the strong (Minnaert) resonance is similar to the one introduced in \cite{AF1} for the bubble-liquid resonance. For the weak resonance, let us assume that 
\[
 \tilde{\Acal}(\omega,\delta)[\Phi](\bx)=\tilde{\Psi}.
\]
According to \eqref{eq:conrew1}, one has that 
\[
  \| \tilde{\Psi} \|_{\mathcal{H}}\ll1.
\]
Set $\Psi=\tilde{\Psi}/ \| \tilde{\Psi} \|_{\mathcal{H}}$ such that $\| {\Psi} \|_{\mathcal{H}}=1$.
If $F$ in \eqref{eq:defAN} is properly chosen which has a component being $\Psi$, one can easily conclude from \eqref{eq:solN} that the scattering wave will blow up at the order $1/\| \tilde{\Psi} \|_{\mathcal{H}}$.

\section{General requirements on the medium configuration and auxiliary results on the layer-potential operators}

In this section, we first introduce some general requirements on the medium configuration that are critical for the occurrence of the Minnaert resonances in our subsequent constructions of the bubble-elastic structures in Sections 4 and 5. Then we derive some auxiliary results for the subsequent use. 

\subsection{General requirements on the medium configuration}
{We will consider}
\begin{equation}\label{eq:as1}
 \delta=\rho_b/\rho_e=o(1), 
\end{equation}
which states that the contrast of the densities of the bubble and the elastic material is high.
Moreover, we assume that the bulk modulus of the air $\kappa$ and the compression modulus $\tilde{\lambda}$ as well as the shear modulus $\tilde{\mu}$ of the elastic material {satisfy}
\begin{equation}\label{eq:as2}
 \kappa/\tilde{\lambda}=\mathcal{O}(\delta) \quad \mbox{and} \quad \tilde{\mu}/\tilde{\lambda}=o(1).
\end{equation}
Under these assumptions, {we can easily derive}
\begin{equation}\label{eq:dtau}
 \tau=\frac{c_b}{\tilde{c}_p}=\frac{\sqrt{\kappa/\rho_b}}   {\sqrt{\left(\tilde{\lambda}+2\tilde{\mu}\right)/\rho_e}}=\mathcal{O}(1),
\end{equation}
where $c_b$ and $\tilde{c}_p$ are defined in \eqref{eq:mo} and \eqref{pa:ksp}, respectively. Indeed, the assumptions in \eqref{eq:as1} and \eqref{eq:as2} are reasonable and this is the case that air bubbles are embedded in the polydimethylsiloxane, a soft elastic material (cf. \cite{CTL}). 

As a matter of fact, the low-frequency is mainly caused by the fact that the size of the air bubble $D$ is much smaller than the wavelength of the elastic wave. Since the elastic wave can be decomposed into the compressional wave (p-wave) and the shear wave (s-wave) \cite{Kup}, we are mainly concerned in this paper with the case that the wavelength of the p-wave is much larger than the size of the bubble $D$ and the wavelength of the s-wave generically does not satisfy this requirement. That means that the Minnaert resonance is mainly caused by the p-wave. Thus by the coordinate transformation, we may assume that the size of the domain $D$ is of order $1$ and $\omega=o(1)$. Since $c_b$ is fixed, {we further have}
\[
 \tilde{k}=o(1)  \quad \mbox{and} \quad \tilde{k}_p= o(1),
\]
where $\tilde{k}$ and $\tilde{k}_p$ are defined in \eqref{eq:mo} and \eqref{pa:ksp}, respectively. 

Let $L$ be the typical length (average length) of the domain $D$. Then we introduce the following non-dimensional parameters:
\begin{equation}\label{eq:pare}
\begin{split}
 &\qquad\bx^{\prime}=\bx/L, \qquad k=\tilde{k} L, \qquad \bu^{\prime}=\bu/L,\\
&\mu =\tilde{\mu}/(\tilde{\lambda} + 2\tilde{\mu}), \quad \lambda=\tilde{\lambda}/(\tilde{\lambda} + 2\tilde{\mu}),\quad u^{\prime}=u/(\rho_b c_b^2).
\end{split}
\end{equation}
Thus from the previous assumptions, one has that
\begin{equation}\label{eq:paas}
k= o(1),\quad \delta=\rho_b/\rho_e=o(1),\quad \tau=\mathcal{O}(1), \quad \mu=o(1) \quad \mbox{and} \quad  \lambda=\mathcal{O}(1).
\end{equation}
Substituting these parameters into equation \eqref{eq:mo} and dropping the primes, one can obtain the following coupled PDE system for our subsequent study:
\begin{equation}\label{eq:mond}
  \left\{
    \begin{array}{ll}
      \mathcal{L}_{\lambda, \mu}\bu(\bx) +k^2 \tau^2\bu(\bx) =0 & \bx\in \mathbb{R}^N\backslash \overline{D},\medskip\\
        \triangle u(\bx) + k^2 u(\bx) =0    & \bx\in D,\medskip \\
   \mathbf{u}(\bx)\cdot\bnu - \frac{1}{ k^2} \nabla u(\bx)\cdot \bnu=0      & \bx\in\partial D, \medskip\\
      \partial_{{\bnu}}\mathbf{u}(\bx) +  \delta \tau^2 u(\bx)\bnu=0 & \bx\in\partial D,\medskip\\
     \bu(\bx)-\bu^i(\bx)  \qquad \mbox{satisfies the radiation condition},
    \end{array}
  \right.
\end{equation} 
where $\tau$ is defined in \eqref{eq:dtau}. Here we would like to point out that in \eqref{eq:mond}, the p-wavenumber satisfies
\[
  k_p=\frac{k \tau}{c_p}=\frac{k \tau}{\sqrt{\lambda+2\mu}}=o(1).
\]

Following our earlier discussions in \eqref{eq:solN}--\eqref{eq:conre1}, 
the solution to the system \eqref{eq:mond} can be given by 
\begin{equation}\label{eq:sol}
  \bu=
 \left\{
   \begin{array}{ll}
     S_{\partial D}^{k}[\varphi_b](\bx), & \bx\in D,\medskip \\
     \bS_{\partial D}^{k\tau}[\bvarphi_e](\bx) +\bu^i, &  \bx\in \mathbb{R}^N\backslash \overline{D},
   \end{array}
 \right.
\end{equation}
for some surface densities $(\varphi_b, \bvarphi_e)\in \mathcal{H}$ that satisfy:
\begin{equation}\label{eq:defA}
\Acal(k,\delta) [\Phi](\bx)=F(\bx), \quad \bx\in\partial D,
\end{equation}
where
\[
\Acal(k,\delta)=  \left(
    \begin{array}{cc}
        \frac{1}{k^2}\left(-\frac{I}{2} + K^{k,*}_{\partial D}\right) &  - \bnu\cdot {\bS}_{\partial D}^{k\tau}\medskip \\
      \delta \tau^2\bnu S^k_{\partial D} & \frac{I}{2} + {\bK}^{k\tau,*}_{\partial D}\\
    \end{array}
  \right),
  \;
  \Phi= \left(
    \begin{array}{c}
      \varphi_b \\
     \bvarphi_e \\
    \end{array}
  \right)
   \; \mbox{and} \;
  F= \left(
    \begin{array}{c}
    \bnu\cdot \bu^i \\
     -\partial_{\bnu}  \bu^i \\
    \end{array}
  \right).
\]
Based on our earlier definitions of the strong and weak Minnaert resonances, we shall establish the sufficient conditions for the occurrence of resonances associated with \eqref{eq:sol}--\eqref{eq:defA}, that is, 
{it holds for the strong resonance of \eqref{eq:mond} that  
\begin{equation}\label{eq:conre}
 \Acal(k,\delta)[\Phi](\bx)=0, 
\end{equation} 
while it holds for the weak resonance of \eqref{eq:mond} that}
\begin{equation}\label{eq:conrew}
 \|\Acal(k,\delta)[\Phi](\bx)\|_{\mathcal{H}}\ll1
\end{equation}
for a nontrivial $\Phi\in\mathcal{H}$ with $\| \Phi\|_{\mathcal{H}}=1$ and $k\ll1$,


\subsection{Some auxiliary results}

We first introduce the following lemmas.
\begin{lem}\label{lem:ul}
 If a vector field $\bw\in H^{1}(\mathbb{R}^3\backslash \overline{D})^3$ satisfies the following three equations:
\[
  \triangle\bw + k^2 \bw=0, \quad \nabla\times\bw=0 \quad \mbox{and} \quad \nabla\cdot\bw=0,
\]
with $k\neq 0$, then $\bw\equiv 0$.
\end{lem}
\begin{proof}
{Direct calculations} show that 
\[
 \nabla\times\nabla\times \bw=\nabla \nabla\cdot \bw -\triangle \bw= 0+k^2 \bw =0.
\]
Thus one can obtain $\bw\equiv 0$ since $k^2\neq 0$.
\end{proof}

%
Recall that the operator $\bS_{\partial D}^{\omega,s}: L^2(\partial D)^3 \rightarrow H^1(\mathbb{R}^3\backslash \overline{D})^3$ is defined in \eqref{eq:sps}. In what follows, if $\bvarphi\in L^2(\partial D)^3$ satisfies 
\[
 \int_{\partial D} \frac{1}{\rho \omega^2}(k_s^2 \mathbf{I} + \nabla \nabla) G^{k_s}(\bx-\by)\bvarphi(\by) ds(\by)=0, \quad \bx\in\mathbb{R}^3\backslash \overline{D},
\]
then we say that $\bvarphi\in \ker\left(\bS_{\partial D}^{\omega,s}\right)$.

\begin{lem}\label{lem:0}
For $\bvarphi\in \ker\left(\bS_{\partial D}^{\omega,s}\right)$, one has that
\[
\int_{\partial D} \nabla \nabla G^{0}(\bx-\by)\bvarphi(\by) ds(\by) =0,  \quad \bx\in\mathbb{R}^3\backslash {D},
\]
where $G^{k}(\bx-\by)$ is defined in \eqref{eq:fu_he} with $k=0$.
\end{lem}

\begin{proof}
From the definition of the fundamental solution in \eqref{eq:ef}, if $\bvarphi\in \ker\left(\bS_{\partial D}^{\omega,s}\right)$, one has that
\[
  \int_{\partial D} \frac{1}{\rho \omega^2}(k_s^2 \mathbf{I} + \nabla \nabla) G^{k_s}(\bx-\by)\bvarphi(\by) ds(\by)=0, \quad \bx\in\mathbb{R}^3\backslash \overline{D}.
\]
Thus one can further have that for $\bx\in\mathbb{R}^3\backslash \overline{D}$, 
\begin{equation}\label{eq:in1}
-\int_{\partial D} \nabla \nabla G^{0}(\bx-\by)\bvarphi(\by) ds(\by) =  \int_{\partial D} \left( k_s^2  G^{k_s} + \nabla \nabla (G^{k_s}-G^0) \right)(\bx-\by)\bvarphi(\by) ds(\by).
\end{equation}
From the expression
\[
G^{0}(\bx-\by)= -\frac{1}{4\pi |\bx-\by|},
\]
the integral possesses the following property
\[
\int_{\partial D} \nabla \nabla G^{0}(\bx-\by)\bvarphi(\by) ds(\by) \rightarrow 0 \quad\mbox{as}\quad |\bx|\rightarrow \infty.
\]
Therefore from the expansion of the fundamental solution $G^{k_s}(\bx-\by)$ and \eqref{eq:in1}, one can obtain that for $\bx\in\mathbb{R}^3\backslash \overline{D}$
\begin{equation}\label{eq:in2}
-\int_{\partial D} \nabla \nabla G^{0}(\bx-\by)\bvarphi(\by) ds(\by) = -\frac{k_s^2}{4\pi}  \int_{\partial D} \left(  \frac{ 1 }{ |\bx-\by|} + \nabla \nabla (|\bx-\by|) \right) \bvarphi(\by) ds(\by).
\end{equation}
Taking the Laplace operator $\triangle$ on both sides of the last equation gives that for $\bx\in\mathbb{R}^3\backslash \overline{D}$
\[
0= -\frac{k_s^2}{4\pi}  \int_{\partial D}  \nabla \nabla \left(  \frac{ 1 }{ |\bx-\by|} \right) \bvarphi(\by) ds(\by).
\]
The proof is completed by noting that the function on the right side of the equation \eqref{eq:in2} is continuous from $\mathbb{R}^3\backslash \overline{D}$ to $\mathbb{R}^3\backslash {D}$.

\end{proof}

\begin{lem}\label{lem:kerss}
If $\bvarphi\in \ker\left(\bS_{\partial D}^{\omega,s}\right)$ does not depend on $k_s$, then one has that
\[
 \bS_{\partial D}^{\omega,p}[\bvarphi](\bx)= \frac{1}{\lambda+2\mu} \int_{\partial D}  G^{k_p}(\bx-\by)\bvarphi(\by) ds(\by), \quad \bx\in\mathbb{R}^3\backslash {D},
\]
where the operators $\bS_{\partial D}^{\omega,i}$ ($i=p, s$) are defined in \eqref{eq:sps}. 
\end{lem}
\begin{proof}
 From the definition of the fundamental solution in \eqref{eq:ef}, if $\bvarphi\in \ker\left(\bS_{\partial D}^{\omega,s}\right)$, one has that
\[
  \int_{\partial D} \frac{1}{\rho \omega^2}(k_s^2 \mathbf{I} + \nabla \nabla) G^{k_s}(\bx-\by)\bvarphi(\by) ds(\by)=0, \quad \bx\in\mathbb{R}^3\backslash \overline{D}.
\]
Replacing $k_s$ with $k_p$ in the last equation yields that for $ \bx\in\mathbb{R}^3\backslash \overline{D}$,
\[
\int_{\partial D} \frac{-1}{\rho \omega^2}(\nabla \nabla) G^{k_p}(\bx-\by)\bvarphi(\by) ds(\by) = \frac{1}{\lambda+2\mu} \int_{\partial D}  G^{k_p}(\bx-\by)\bvarphi(\by) ds(\by).
\]
 The proof is readily completed by noting that the operator $S_{\partial D}^{k_p}$ is continuous from $\mathbb{R}^3\backslash \overline{D}$ to $\mathbb{R}^3\backslash {D}$.

\end{proof}

\begin{rem}
{We can derive} that $\ker\left(\bS_{\partial D}^{\omega,s}\right)\neq \emptyset$ and $\bnu\in \ker\left(\bS_{\partial D}^{\omega,s}\right)$. Indeed, we set 
\begin{equation}\label{eq:bw}
 \bw(\bx)=\int_{\partial D} \frac{1}{\rho \omega^2}(k_s^2 \mathbf{I} + \nabla \nabla) G^{k_s}(\bx-\by)\bnu_\by ds(\by), \quad \bx\in \mathbb{R}^3\backslash \overline{D}.
\end{equation}
It is directly verified that the function $\bw$ defined in \eqref{eq:bw} satisfies the following two equations in $\mathbb{R}^3\backslash \overline{D}$:
\[
 \triangle \bw +k_s^2 \bw=0 \quad \mbox{and} \quad \nabla\cdot\bw=0.
\]
With the help the identities 
\[
 \nabla\times \nabla =0 \quad \mbox{and} \quad  \nabla_\bx G^{k_s}(\bx-\by) = -\nabla_\by G^{k_s}(\bx-\by),
\]
one furthermore has for $\bx\in \mathbb{R}^3\backslash \overline{D}$ that 
\[
\begin{split}
 \nabla\times\bw &=\int_{\partial D} \frac{1}{\rho \omega^2} \nabla_\bx\times(k_s^2 \mathbf{I} + \nabla \nabla) G^{k_s}(\bx-\by)\bnu_\by ds(\by) \\
 & = \frac{1}{\rho \omega^2}\int_{\partial D} (k_s^2 \nabla_\bx\times \mathbf{I} + \nabla_\bx\times\nabla \nabla) G^{k_s}(\bx-\by)\bnu_\by ds(\by) \\
 & = \frac{-k_s^2 }{\rho \omega^2}\int_{\partial D} \nabla_\by G^{k_s}(\bx-\by)\times \bnu_\by ds(\by)  \\
& = \frac{k_s^2 }{\rho \omega^2}\int_{ D} \nabla_\by\times\nabla_\by G^{k_s}(\bx-\by) d\by  =0.
\end{split}
\]
Finally, Lemma \ref{lem:ul} shows that $\bw$ defined in \eqref{eq:bw} vanishes in $\mathbb{R}^3\backslash \overline{D}$ and one can conclude that $\bnu\in \ker\left(\bS_{\partial D}^{\omega,s}\right)$ and  $\bnu$ does not depend on $k_s$.
\end{rem}

\begin{lem}\label{lem:ash}
For the operators $S_{\partial D}^{k}: L^2(\partial D) \rightarrow H^1(\partial D)$ and $K_{\partial D}^{k,*}: L^2(\partial D)\rightarrow L^2(\partial D)$ defined in \eqref{eq:s_h} and \eqref{eq:ju_he}, respectively, 
we have the following asymptotic expansions {in three dimensions {(cf. \cite{AF1})}:
\begin{equation}\label{eq:hsa}
 S_{\partial D}^{k} =\sum_{j=0}^{+\infty} k^j  S_{\partial D,j} , \quad 
 K_{\partial D}^{k,*}  =\sum_{j=0}^{+\infty} k^j  K_{\partial D,j}^{*} 
\end{equation}
}
where
\[
  S_{\partial D,j}[\varphi](\bx)=-\frac{\rmi}{4\pi}\int_{\partial D}\frac{(\rmi|\bx-\by|)^{j-1}}{j!}\varphi(\by)ds(\by),
\]
and 
\[
 K_{\partial D,j}^{*}[\varphi](\bx)=-\frac{\rmi^j(j-1)}{4\pi j!}\int_{\partial D} |\bx-\by|^{j-3}(\bx-\by)\cdot\bnu_\bx \varphi(\by) ds(\by).
\]
{Moreover, $S_{\partial D,j}$ and $K_{\partial D,j}^{*} $are uniformly bounded with respect to $j$, 
and the two series in \eqref{eq:hsa} are convergent in $\mathcal{L}(L^2(\partial D), H^1(\partial D))$ 
and $\mathcal{L}(L^2(\partial D))$, respectively. 
}
\end{lem}

As discussed earlier, only the wavelength of the p-wave is required to be asymptotically larger than the size of the domain $D$ and the wavelength of the s-wave is not required to satisfy such a requirement, thus the low-frequency resonance is mainly caused by the p-wave and the s-wave generically makes no contribution. Therefore, in the following analysis for the low-frequency resonance, we choose to mainly consider the density function $\bvarphi\in \ker\left(\bS_{\partial D}^{\omega,s}\right)$, which can be proved not depending on $k_s$ later. From Lemma \ref{lem:kerss}, we can focus ourselves on the operator $\bS_{\partial D}^{\omega,p}$ with the kernel $\delta_{ij} G^{k_p}/(\lambda+2\mu)$ and the operator $\bK_{\partial D}^{\omega,p,*}: L^2(\partial D)^3 \rightarrow L^2(\partial D)^3$ defined in \eqref{eq:lanp} with the kernel function $\bGa^{\omega}$ replaced by $\delta_{ij} G^{k_p}/(\lambda+2\mu)$. By straightforward calculations, we have the following asymptotic expansions for the operators $\bS_{\partial D}^{\omega,p}$ and $\bK_{\partial D}^{\omega,p,*}$ {for the density function $\bvarphi\in \ker\left(\bS_{\partial D}^{\omega,s}\right)$} that does not depend on $k_s$.
\begin{lem}\label{lem:asl}
{For the density function} $\bvarphi\in \ker\left(\bS_{\partial D}^{\omega,s}\right)$ does not depend on $k_s$, 
the operators $\bS_{\partial D}^{\omega,p}$ from $L^2(\partial D)^3$ to $H^1(\partial D)^3$ and $\bK_{\partial D}^{\omega,p,*}$ from $L^2(\partial D)^3$ to $L^2(\partial D)^3$ enjoy the following asymptotic expansions {in three dimensions:
\begin{equation}\label{eq:lsa}
 \bS_{\partial D}^{\omega,p} =  \sum_{j=0}^{+\infty} k_p^j  \bS_{\partial D,j}^{p} , \quad 
 \bK_{\partial D}^{\omega,p,*} = \sum_{j=0}^{+\infty} k_p^j \bK_{\partial D,j}^{p,*} ,
\end{equation}
}
where
\[
  \bS_{\partial D,j}^{p}[\bvarphi](\bx)=-\frac{\rmi}{4\pi}\int_{\partial D}\frac{(\rmi|\bx-\by|)^{j-1}}{j!}\bvarphi(\by)ds(\by),
\]
and 
\begin{equation}\label{eq:apk}
 \bK_{\partial D,j}^{p,*} [\bvarphi](\bx)=\frac{\lambda}{\lambda+2\mu}\mathbf{R}_{1,j} [\bvarphi](\bx) + \frac{\mu}{\lambda+2\mu} \mathbf{R}_{2,j} [\bvarphi](\bx), 
\end{equation}
{with $\mathbf{R}_{1,j} $ and $\mathbf{R}_{2,j}$ given by}
\[
 \mathbf{R}_{1,j} [\bvarphi](\bx) =-\frac{\rmi^j(j-1)\bnu_\bx}{4\pi j!}\int_{\partial D} |\bx-\by|^{j-3}\langle\bx-\by,\bvarphi(\by)\rangle ds(\by), 
\]
and 
\[
\begin{split}
 \mathbf{R}_{2,j} [\bvarphi](\bx) =&-\frac{\rmi^j(j-1)}{4\pi j!} \left( \int_{\partial D} |\bx-\by|^{j-3}\langle\bx-\by, \bnu_\bx\rangle \bvarphi(\by) ds(\by) +\right. \\
& \qquad \qquad \left. \int_{\partial D} |\bx-\by|^{j-3}(\bx-\by) \langle \bnu_\bx, \bvarphi(\by)\rangle ds(\by)  \right) .
\end{split}
\]
{Moreover, $\bS_{\partial D,j}^{p}$ and $\bK_{\partial D,j}^{p,*}$ are 
uniformly bounded with respect to $j$, and the two series in \eqref{eq:lsa} are convergent in $\mathcal{L}(L^2(\partial D)^3, H^1(\partial D)^3)$ and $\mathcal{L}(L^2(\partial D)^3)$, respectively.
}
\end{lem}

\begin{lem}\label{lem:ker1}
If $\bvarphi\in \ker\left(\bS_{\partial D}^{\omega,s}\right)$, $k_p\ll 1$,  $\mu\ll1$ and $\lambda=\mathcal{O}(1)$, then one has that for $\bx\in\partial {D}$,
\[
   {\bK}^{\omega,*}_{\partial D}[\bvarphi](\bx) = \mathbf{R}_{1,0}[\bvarphi] + \mathcal{O}(\mu) + \mathcal{O}(k_p^2),
\]
where $\mathbf{R}_{1,0}$ is defined in \eqref{eq:apk}.
\end{lem}
\begin{proof}
From the definition of the fundamental solution in \eqref{eq:ef} and the fact $\bvarphi\in \ker\left(\bS_{\partial D}^{\omega,s}\right)$, we only need to deal with the kernel function $-\frac{1}{\omega^2} \nabla \nabla G^{{k}_p}$. Moreover, the traction operator $\partial_{{\bnu}}$ defined in \eqref{eq:trac} can also be written as 
\begin{equation}\label{eq:redt}
\partial_{{\bnu}} \bw = 2\mu \nabla\bw\cdot \bnu + \lambda (\nabla\cdot\bu) \bnu +\mu \bnu\times (\nabla\times \bw).
\end{equation}
From Lemma \ref{lem:0}, has that 
\begin{equation}\label{eq:in3}
\begin{split}
 &\int_{\partial D} 2\mu \nabla \left(-\frac{1}{\omega^2} \nabla \nabla G^{{k}_p}(\bx-\by) \bvarphi(\by)\right)\cdot\bnu ds(\by) \\ 
= &\int_{\partial D} 2\mu \nabla \left(-\frac{1}{\omega^2} \nabla \nabla \left(G^{{k}_p} - G^0 \right)(\bx-\by) \bvarphi(\by)\right)\cdot\bnu ds(\by) \\
=& \mathcal{O}(\mu),
\end{split}
\end{equation}
where the last identity follows from the Lemma \ref{lem:asl}. There also holds that 
\begin{equation}\label{eq:in4}
\begin{split}
 & \lambda \bnu\int_{\partial D} \nabla \cdot \left(-\frac{1}{\omega^2} \nabla \nabla G^{{k}_p}(\bx-\by) \bvarphi(\by)\right) ds(\by) \\
 = & \lambda \bnu\int_{\partial D} -\frac{1}{ \omega^2} \triangle \nabla G^{{k}_p}(\bx-\by)\cdot \bvarphi(\by)ds(\by)  \\
  = &  \bnu\int_{\partial D}  \nabla G^{{k}_p}(\bx-\by)\cdot \bvarphi(\by)ds(\by)  \\
   = & \mathbf{R}_{1,0}[ \bvarphi(\by)]+  \mathcal{O}(k_p^2).
\end{split}
\end{equation}
where $\mathbf{R}_{1,0}$ is defined in \eqref{eq:apk}. Moreover, since $\nabla\times\nabla =0$, one finally concludes that from \eqref{eq:redt}, \eqref{eq:in3} and \eqref{eq:in4} 
\[
 {\bK}^{\omega,*}_{\partial D}[\bvarphi](\bx) = \mathbf{R}_{1,0}[\bvarphi] + \mathcal{O}(\mu) + \mathcal{O}(k_p^2).
\]
The proof is completed.
\end{proof}

\begin{rem}
Lemma \ref{lem:ker1} holds for any $\bvarphi\in \ker\left(\bS_{\partial D}^{\omega,s}\right)$, which could depend on $k_s$.
\end{rem}

{For the later convenience, we introduce an important subspace of  {$L^2(\partial D)$:}}
\begin{equation}\label{eq:l0}
 L^2_0(\partial D)=\{\varphi\in L^2(\partial D) : \int_{\partial D} \varphi ds=0 \},
\end{equation}
and the following results that can be found in \cite{HK07:book}.
\begin{lem}\label{lem:ks1}
Let $\xi$ be a real number. The operator $\xi-K^{*}_{\partial D,0}$ is invertible on $L^2_0(\partial D)$ if $|\xi|\geq 1/2$, 
where 
$K^{*}_{\partial D,0}$ is given in \eqref{eq:dekk}. Furthermore, 
the kernel of the operator $\left(-\frac{I}{2} + K^{*}_{\partial D,0}\right) $, 
{restricted in the space $L^2(\partial D)$,} 
is one dimensional, and 
\[
 \ker \left(-\frac{I}{2} + K^{*}_{\partial D,0}\right) =\spn\{S_{\partial D,0}^{-1}[1]\},
\]
where the operator $S_{\partial D,0}^{-1}$ is the inverse of the operator $S_{\partial D,0}$ defined in \eqref{eq:s_h}.
\end{lem}

\begin{lem}\label{lem:k}
{All $f\in L^2(\partial D)$ satisfying $\left(-\frac{I}{2} +  K_{\partial D,0}\right)f=0$, with $K_{\partial D,0}$ defined in \eqref{eq:dekk}, are constant.}
\end{lem}

%

\section{Minnaert resonances in three dimensions}

In this section, we show the Minnaert resonances for the system \eqref{eq:mond} in $\mathbb{R}^3$. 
{We first prove
that the weak resonance always} occurs provided that the parameters are properly chosen. 
Then with a proper choice of the geometry of the domain $D$, we further show that 
enhanced or even strong resonances can occur.

\begin{thm}\label{thm:wr1}
Consider the system \eqref{eq:mond} in three dimensions. If the parameters are chosen according to \eqref{eq:as1}--\eqref{eq:as2} (or equivalently \eqref{eq:paas}), {then weak Minnaert resonance occurs.}
\end{thm}
\begin{proof}
The proof is proceeded by construction. By the definition of the weak resonance in \eqref{eq:conrew} for the system \eqref{eq:mond}, we construct in what follows a density function $\Phi\in\mathcal{H}$ with $\| \Phi\|_{\mathcal{H}}=1$ such that the condition \eqref{eq:conrew} is fulfilled. 

Since we consider the low-frequency resonance, i.e. $k\ll 1$, then the density function $\Phi\in\mathcal{H}$ should satisfy the following asymptotic expansion
\[
\Phi=\Phi_0 + k \Phi_1 + k^2 \Phi_2 + \cdots, 
\]
where
\[
 \Phi_j=
  \left(
    \begin{array}{c}
      \varphi_j \\
     \bvarphi_j \\
    \end{array}
  \right), \mm{~~j=0, 1, 2, \cdots}
\]

As we discussed earlier, the resonance is mainly caused by the p-waver in our study. Therefore, we choose 
\[
 \bvarphi_j \in \ker\left(\bS_{\partial D}^{\omega,s}\right), \mm{~~j=0, 1, 2, \cdots}.
\]
\rr{From the assumption that $\delta\ll1$ and the operator $S^k_{\partial D}$ is bounded},
\mm{we can derive from the definition of the operator $\Acal(k,\delta)$ in \eqref{eq:defA} and Lemmas \ref{lem:ash}-\ref{lem:ker1} that
}
\begin{equation}\label{eq:soe1}
\begin{split}
 \Acal(k,\delta)[\Phi] = & \frac{1}{k^2} \left(
    \begin{array}{c}
      \left(-\frac{I}{2} + K^{*}_{\partial D,0}\right) [\varphi_0] \\
     0 \\
    \end{array}
  \right)
 +   \frac{1}{k} \left(
    \begin{array}{c}
      \left(-\frac{I}{2} + K^{*}_{\partial D,0}\right) [\varphi_1] \\
     0 \\
    \end{array}
  \right) +  \\
  & 
 \left(
    \begin{array}{c}
       K^{*}_{\partial D,2}[\varphi_0] -  \bnu\cdot {\bS}_{\partial D}^{k\tau}[\bvarphi_0] + \left(-\frac{I}{2} + K^{*}_{\partial D,0}\right) [\varphi_2]\\
     \left(\frac{I}{2} +  {\bK}^{k\tau,*}_{\partial D}\right)[\bvarphi_0] \\
    \end{array}
  \right) 
+\mathcal{O}(k)+\rr{\mathcal{O}(\delta)}.
\end{split} 
\end{equation}
Since $k\ll1$, the first two terms in \eqref{eq:soe1} should vanish. Thus one can conclude from Lemma \ref{lem:ks1} that
\begin{equation}\label{eq:v0v1}
 \varphi_0, \varphi_1 \in \ker \left(-\frac{I}{2} + K^{*}_{\partial D,0}\right) =\spn\{S_{\partial D,0}^{-1}[1]\}.
\end{equation}
Next we deal with the third term in \eqref{eq:soe1}. Since $\bvarphi_0 \in \ker\left(\bS_{\partial D}^{\omega,s}\right)$, one has that form Lemma \ref{lem:ker1}
\[
 \left(\frac{I}{2} +  {\bK}^{k\tau,*}_{\partial D}\right)[\bvarphi_0]  = \frac{1}{2} \bvarphi_0 + \mathbf{R}_{1,0}[\bvarphi_0] + \mathcal{O}(\mu) + \mathcal{O}(k_p^2),
\]
where
\begin{equation}\label{eq:r10}
 \mathbf{R}_{1,0}[\bvarphi_0] =\bnu_{\bx}\int_{\partial D} \frac{\langle\bx-\by,\bvarphi_0\rangle}{4\pi |\bx-\by|^3}ds(\by),
\end{equation}
and is bounded from $L^2(\partial D)^3$ to $L^2(\partial D)^3$. Therefore the leading term $ \frac{1}{2} \bvarphi_0 + \mathbf{R}_{1,0}[\bvarphi_0] $ should vanish.  From \eqref{eq:r10}, $\mathbf{R}_{1,0}[\bvarphi_0]$ contains only the normal component, therefore the function $\bvarphi_0$ should also contain only the normal component, namely 
\[
 \bvarphi_0=\varphi \bnu, 
\]
for some $\varphi\in L^2(\partial D)$. Thus 
\[
\frac{1}{2}\bvarphi_0 + \mathbf{R}_{1,0}[\bvarphi_0] =\bnu\left( \frac{1}{2} \varphi - K_{\partial D,0} [\varphi]\right)
\]
and $\varphi$ should be a constant thanks to Lemma \ref{lem:k}. Hence one finally obtains that 
\[
 \bvarphi_0=c_0 \bnu
\]
for some constant $c_0$, which will be further determined later. Since $\bvarphi_0$ derived in the last equation does not depend on $k_s$, one has the following expansion from Lemma \ref{lem:asl}
\[
{\bS}_{\partial D}^{k\tau}[\bvarphi_0] = \bS_{\partial D}^{\omega,p} [\bvarphi_0]=  \sum_{j=0}^{+\infty} k_p^j  \bS_{\partial D,j}^{p} [\bvarphi_0], 
\]
and 
\[
 {\bK}^{k\tau,*}_{\partial D}[\bvarphi_0] = \bK_{\partial D}^{\omega,p,*}[\bvarphi_0]  =\sum_{j=0}^{+\infty} k_p^j \bK_{\partial D,j}^{p,*} [\bvarphi_0].
\]
We proceed to deal with the first component of the third term in \eqref{eq:soe1} by solving the following equation
\begin{equation}\label{eq:var2}
  \left(-\frac{I}{2} + K^{*}_{\partial D,0}\right) [\varphi_2]=  c_0\bnu\cdot {\bS}_{\partial D,0}^{p}[\bnu] - K^{*}_{\partial D,2} [\varphi_0]  .
\end{equation}
Lemma \ref{lem:ks1} shows that the operator $\left(-\frac{I}{2} + K^{*}_{\partial D}\right)$ is invertible on  $L^2_0(\partial D)$, 
thus $c_0$ should be chosen as 
\begin{equation}\label{eq:c0}
 c_0=\frac{\int_{\partial D}  K^{*}_{\partial D,2} [\varphi_0](\bx)  ds(\bx) }{\int_{\partial D} \bnu\cdot {\bS}_{\partial D,0}^{p}[\bnu](\bx)ds(\bx) }
\end{equation}
such that the equation \eqref{eq:var2} is solvable. Thus 
\[
 \varphi_2 =   \left(-\frac{I}{2} + K^{*}_{\partial D,0}\right)^{-1} \left[c_0\bnu\cdot {\bS}_{\partial D,0}^{p}[\bnu] - K^{*}_{\partial D,2} [\varphi_0] \right] .
\]

 Following the same process above, one can construct
\begin{equation}\label{eq:phi}
 \Phi=
  \left(
    \begin{array}{c}
      \varphi_0 \\
     c_0 \bnu \\
    \end{array}
  \right) + 
k\left(
    \begin{array}{c}
      \varphi_1  \\
     c_1 \bnu \\
    \end{array}
  \right) + 
 k^2\left(
    \begin{array}{c}
      \varphi_2  \\
    c_2 \bnu\\
    \end{array}
  \right)
+ 
 k^3\left(
    \begin{array}{c}
      \varphi_3  \\
     0\\
    \end{array}
  \right) + 
 k^4\left(
    \begin{array}{c}
      \varphi_4  \\
     0\\
    \end{array}
  \right),
\end{equation}
where $\varphi_0, \varphi_1$ are given in \eqref{eq:v0v1}, $c_0$ is given \eqref{eq:c0} and
\[
 c_1=\frac{\int_{\partial D}  K^{*}_{\partial D,2} [\varphi_1](\bx) + K^{*}_{\partial D,3} [\varphi_0](\bx) - (\tau/c_p) \bnu\cdot {\bS}_{\partial D,1}^{p}[c_0\bnu](\bx) d\bx    }{\int_{\partial D} \bnu\cdot {\bS}_{\partial D,0}^{p}[\bnu](\bx)d\bx },
\]
\[
 \varphi_3 =   \left(-\frac{I}{2} + K^{*}_{\partial D,0}\right)^{-1} \left[c_1\bnu\cdot {\bS}_{\partial D,0}^{p}[\bnu] - K^{*}_{\partial D,2} [\varphi_1] +(\tau/c_p) c_0\bnu\cdot {\bS}_{\partial D,1}^{p}[\bnu] - K^{*}_{\partial D,3} [\varphi_0] \right] ,
\]
\[
\begin{split}
 c_2= & \frac{\int_{\partial D}  K^{*}_{\partial D,2} [\varphi_2](\bx) + K^{*}_{\partial D,3} [\varphi_1](\bx) +  K^{*}_{\partial D,4} [\varphi_0](\bx) d\bx    }{\int_{\partial D} \bnu\cdot {\bS}_{\partial D,0}^{p}[\bnu](\bx)d\bx } - \\
   & \frac{\int_{\partial D}  (\tau/c_p) \bnu\cdot {\bS}_{\partial D,1}^{p}[c_1\bnu](\bx) + (\tau/c_p)^2 \bnu\cdot {\bS}_{\partial D,2}^{p}[c_0\bnu](\bx)d\bx    }{\int_{\partial D} \bnu\cdot {\bS}_{\partial D,0}^{p}[\bnu](\bx)d\bx },
\end{split}
\]
\[
\begin{split}
 \varphi_4 =   \left(-\frac{I}{2} + K^{*}_{\partial D,0}\right)^{-1} & \left[  c_2\bnu\cdot {\bS}_{\partial D,0}^{p}[\bnu] -K^{*}_{\partial D,2} [\varphi_2](\bx) - K^{*}_{\partial D,3} [\varphi_1](\bx) + \right. \\
 & \left. (\tau/c_p)\bnu\cdot {\bS}_{\partial D,1}^{p}[c_1\bnu](\bx) -  K^{*}_{\partial D,4} [\varphi_0](\bx) + (\tau/c_p)^2 \bnu\cdot {\bS}_{\partial D,2}^{p}[c_0\bnu](\bx) \right] .
\end{split}
\]
Then one can have that
\begin{equation}\label{eq:esap}
\begin{split}
  \Acal(\omega,\delta)[\Phi]= &\left(
    \begin{array}{c}
      0  \\
     \delta \tau^2 \bnu  + \mu\left(\frac{1}{\lambda+2\mu} (\mathbf{R}_{2,0}-2\mathbf{R}_{1,0})[c_0\bnu]\right) +  k^2(\tau/c_p)^2  \mathbf{R}_{1,2}[c_0\bnu]  \\
    \end{array}
  \right) +\\
&\left(
    \begin{array}{c}
      \mathcal{O}(k^3)  \\
      \mathcal{O}(k^3) + \mathcal{O}(\delta k) + \mathcal{O}(\mu k)  \\
    \end{array}
  \right),
\end{split}
\end{equation}
with $\Phi$ defined in \eqref{eq:phi} and the operator $\mathbf{R}_{i,j}$ defined in Lemma \ref{lem:asl}.
Finally one can conclude that 
\[
 \|\Acal(\omega,\delta)[\Phi](\bx)\|_{\mathcal{H}}=\mathcal{O}(\delta) + \mathcal{O}(\mu) +\mathcal{O}(k^2)\ll 1,
\]
which clearly shows that the weak resonance occurs.
%
\end{proof}

From the proof of Theorem \ref{thm:wr1}, one readily sees that \mm{we can not enhance the resonance
by diminishing the parameter $k$ only}. The parameter $k$ should be chosen in an appropriate way that is correlated to the parameters $\delta$ and $\mu$ in order to achieve enhanced resonance effects. 
\mm{In fact, we have the following results.}
\begin{prop}\label{pro:br1}
Consider \mm{the same} setup as that in Theorem \ref{thm:wr1}. If the following equation is solvable 
\begin{equation}\label{eq:br1}
\delta \tau^2 \bnu  + \mu\left(\frac{1}{\lambda+2\mu} (\mathbf{R}_{2,0}-2\mathbf{R}_{1,0})[c_0\bnu]\right) +  k^2(\tau/c_p)^2  \mathbf{R}_{1,2}[c_0\bnu]  =0,
\end{equation}
where $\mathbf{R}_{2,0}$, $\mathbf{R}_{1,0}$ and $\mathbf{R}_{1,2}$ are defined in Lemma \ref{lem:asl},
then one has that 
\begin{equation}\label{eq:esbr}
  \|\Acal(\omega,\delta)[\Phi](\bx)\|_{\mathcal{H}}=\mathcal{O}(\delta k) + \mathcal{O}(\mu k) +\mathcal{O}(k^3),
\end{equation}
which indicates that the enhanced resonance can be achieved. If \eqref{eq:br1} is solvable, the parameter $k$ should fulfil 
\begin{equation}\label{eq:lk}
k=\sqrt{\mathcal{O}(\delta) + \mathcal{O}(\mu) }.
\end{equation}
\end{prop}

\begin{proof}
If we take the density function $\Phi$ as in \eqref{eq:phi}, and the equation in \eqref{eq:br1} is solvable, then from \eqref{eq:esap}, one has that
\begin{equation}
  \Acal(\omega)[\Phi]= \left(
    \begin{array}{c}
      \mathcal{O}(k^3)  \\
      \mathcal{O}(k^3) + \mathcal{O}(\delta k) + \mathcal{O}(\mu k)  \\
    \end{array}
  \right).
\end{equation}
Thus the estimate in \eqref{eq:esbr} is proved. Moreover, by noting that the functions
\[
  (\mathbf{R}_{2,0}-2\mathbf{R}_{1,0})[c_0\bnu]  \quad \mbox{and} \quad  \mathbf{R}_{1,2}[c_0\bnu] 
\]
are bounded in $L^2(\partial D)^3$ and 
\[
 \tau=\mathcal{O}(1) \quad \mbox{and} \quad  \lambda=\mathcal{O}(1),
\]
one can show by direct computations that if the equation \eqref{eq:br1} is solvable then the parameter $k$ fulfils 
\[
 k=\sqrt{\mathcal{O}(\delta) + \mathcal{O}(\mu) }.
\]
\end{proof}

\begin{rem}\label{rem:onu}
If the equation \eqref{eq:br1} is solvable, 
\mm{the function $\mathbf{R}_{2,0}[\bnu]$ should also contain only the normal component, 
since the functions $\mathbf{R}_{1,0}[\bnu]$ and $\mathbf{R}_{1,2}[\bnu]$ contain only the normal components}. Indeed, this is also physically justifiable. \mm{We can see from the fourth equation in \eqref{eq:mond} that 
it is natural to require the leading term of the traction of the elastic wave outside the bubble $D$ to contain 
only the normal component in order to strengthen the resonance, since the pressure in the bubble has only the normal component}. This property depends heavily on the geometry of the domain $D$.
\end{rem}

\begin{rem}
Since $\delta\ll1$ and $\mu\ll1$, \mm{we can readily see from \eqref{eq:esbr} and \eqref{eq:lk} in Proposition \ref{pro:br1}
that enhanced resonance effects can be achieved.}
\end{rem}

\begin{rem}
We apply Proposition \ref{pro:br1} to the case when the bubble $D$ is a unit ball. 
\mm{In such a case, one has that for $\bx\in\partial D$,}
\[
 S_{\partial D,0}^{-1}[1](\bx)=-1 \quad \mbox{and} \quad S_{\partial D,0}[\bnu](\bx)=-\frac{1}{3}\bnu.
\] 
Thus we get from \eqref{eq:v0v1} that
\[
 \varphi_0=\varphi_1 =-1.
\]
We first calculate the parameter $c_0$ defined in  \eqref{eq:c0}. A direct calculation shows that 
\[
 \int_{\partial D}  K^{*}_{\partial D,2} [-1](\bx)  ds(\bx)= -4\pi/3.
\]
Hence, one obtains 
\[
c_0=\frac{\int_{\partial D}  K^{*}_{\partial D,2} [\varphi_0](\bx)  ds(\bx) }{\int_{\partial D} \bnu\cdot {\bS}_{\partial D,0}^{p}[\bnu](\bx)ds(\bx) }=\frac{-4\pi/3}{-4\pi/(3(\lambda+2\mu))}=\lambda+2\mu.
\]
From the proof in Theorem \ref{thm:wr1}, we derive 
\[
  \mathbf{R}_{1,0}[\bnu]=-\bnu/2.
\]
Moreover, by some straightforward but rather tedious calculations, one can obtain that 
\[
 \mathbf{R}_{2,0}[\bnu]=\bnu/3 \quad \mbox{and} \quad  \mathbf{R}_{1,2}[\bnu]=-\bnu/3.
\]
Therefore the equation \eqref{eq:br1} in Proposition \ref{pro:br1} can be simplified to be 
\[
 \delta \tau^2 \bnu  +\frac{4}{3} \mu \bnu-  \frac{1}{3} k^2\tau^2\bnu =0,
\]
which shows that $k$ should be chosen as 
\begin{equation}\label{eq:re3}
  k=\sqrt{3\delta +4\mu/\tau^2}.
\end{equation}
Substituting the parameters in \eqref{eq:pare} into the last equation shows that the resonance frequency should be 
\[
 \omega=\sqrt{\frac{3\kappa + 4\tilde{\mu}}{\rho_e}},
\]
which recovers the physical result in \cite{AR}. Definitely, \eqref{eq:br1} may be solvable in more general scenarios, which then yield bubble-elastic structures that can produce enhanced resonances. 
\end{rem}

As also discussed earlier, the geometry of the bubble shall also play a critical role in the resonance in our study. We next show that if the domain $D$ is properly chosen, then resonance effects can also be significantly enhanced.

\begin{prop}\label{pro:br2}
Consider the same setup as that in Theorem \ref{thm:wr1}. Furthermore, if the following equation is solvable,
  \begin{equation}\label{eq:br2}
\begin{split}
   &  \delta \bnu \sum_{j=0}^{m}\tau^2 k^j \sum_{i=0}^j  S_{\partial D,i}[\varphi_{j-i}] +\frac{k^2 \tau^2}{c_p^2} \sum_{j=0}^{m}\left(\frac{k\tau}{c_p}\right)^j \sum_{i=2}^{j+2}  \mathbf{R}_{1,i}[c_{j+2-i}\bnu] + \\ 
& \frac{\mu}{\lambda+2\mu} \sum_{j=0}^{m} \left(\frac{k\tau}{c_p}\right)^j \sum_{i=0}^j   (\mathbf{R}_{2,i}-2\mathbf{R}_{1,i})[c_{j-i}\bnu]=0, 
\end{split}
\end{equation}
where $\varphi_j$ is defined in \eqref{eq:wp1}, then one has that 
\begin{equation}
  \|\Acal(\omega,\delta)[\Phi](\bx)\|_{\mathcal{H}}=\mathcal{O}(\delta k^{m+1}) + \mathcal{O}(\mu k^{m+1}) +\mathcal{O}(k^{m+3}).
\end{equation}
\end{prop}

\begin{proof}
Following the proof of Theorem \ref{thm:wr1}, one can construct 
\begin{equation}\label{eq:wp1}
 \Phi=\sum_{j=0}^\infty k^j
  \left(
    \begin{array}{c}
      \varphi_j \\
      \bvarphi_j  \\
    \end{array}
  \right) 
\end{equation}
where $\varphi_0, \varphi_1$ are the same as those in \eqref{eq:v0v1} and 
\[
  \varphi_j=\left(-\frac{I}{2} + K^{*}_{\partial D}\right)^{-1} \left[ \sum_{m=0}^{j-2} (\tau/c_p)^m \bnu\cdot {\bS}_{\partial D,m}^{p}[\bvarphi_{j-2-m}] -\sum_{m=2}^j K^{*}_{\partial D,m} [\varphi_{j-m}] \right] \; \mbox{for} \; j\geq 2,
\]
\[
 \bvarphi_j =c_j \bnu, \quad \mbox{for} \quad j\geq 0, 
\]
with 
\[
  c_j =\frac{\int_{\partial D} \sum_{m=2}^{j+2} K^{*}_{\partial D,m} [\varphi_{j+2-m}](\bx) - \sum_{m=1}^{j} (\tau/c_p)^m \bnu\cdot {\bS}_{\partial D,m}^{p}[\bvarphi_{j-m}](\bx) ds(\bx)}{\int_{\partial D} \bnu\cdot {\bS}_{\partial D,0}^{p}[\bnu] ds(\bx)}.
\]
It is remarked here that when calculating $\Phi$ in \eqref{eq:wp1}, one should first calculate $c_j$ to obtain $\bvarphi_j$ and then calculate $\varphi_{j+2}$ for $j=0, 1, 2,\ldots$, since $\varphi_0, \varphi_1 \in \spn\{S_{\partial D,0}^{-1}[1]\}$.
Hence one has that 
\[
  \Acal(\omega,\delta)[\Phi]_1=0,
\]
\[
\begin{split}
  \Acal(\omega,\delta)[\Phi]_2= &  \delta \bnu \sum_{j=0}^{\infty}\tau^2 k^j \sum_{i=0}^j  S_{\partial D,i}[\varphi_{j-i}] + \frac{k^2 \tau^2}{c_p^2} \sum_{j=0}^{\infty}\left(\frac{k\tau}{c_p}\right)^j \sum_{i=2}^{j+2}  \mathbf{R}_{1,i}[c_{j+2-i}\bnu] + \\ 
& \frac{\mu}{\lambda+2\mu} \sum_{j=0}^{\infty} \left(\frac{k\tau}{c_p}\right)^j \sum_{i=0}^j   (\mathbf{R}_{2,i}-2\mathbf{R}_{1,i})[c_{j-i}\bnu],
\end{split}
\]
where $\Acal(\omega,\delta)[\Phi]_i$ denotes the $i$-th component of the vectorial function  $\Acal(\omega,\delta)[\Phi]$ and the operators $\mathbf{R}_{i,j}$ with $i=1,2$ and $j\geq 0$, are defined in \eqref{eq:apk}. Thus, if the equation \eqref{eq:br2} is solvable, one can conclude that 
\begin{equation*}
  \|\Acal(\omega,\delta)[\Phi](\bx)\|_{\mathcal{H}}=\mathcal{O}(\delta k^{m+1}) + \mathcal{O}(\mu k^{m+1}) +\mathcal{O}(k^{m+3}).
\end{equation*}
%
\end{proof}

\begin{rem}
It is noted that if the equation \eqref{eq:br2} is solvable, then one should have 
\[
 k=\sqrt{\mathcal{O}(\delta) + \mathcal{O}(\mu) }
\]
and
\begin{equation}\label{eq:rbnu}
 \mathbf{R}_{2,i}[\bnu] =\psi_i \bnu \quad \mbox{for} \quad 0\leq i\leq m,
\end{equation}
with $\psi_i \in L^2(\partial D)$. \mm{The identities \eqref{eq:rbnu} are unobjectionably} reasonable as explained in Remark \ref{rem:onu}.
\end{rem}

\begin{rem}
Proposition \ref{pro:br1} is a special case of Proposition \ref{pro:br2} with m=0. Indeed, even though the equation \eqref{eq:br2} could be solved for $m>0$, it is enough to solve the equation \eqref{eq:br2} with $m=0$, namely the equation \eqref{eq:br1} in Proposition \ref{pro:br1}, to obtain the resonant frequency. This is because that it gives the leading-order term of the resonant frequency.
\end{rem}
\begin{rem}
If the equation \eqref{eq:br2} is solvable for $m=\infty$, then the function $\Phi$ defined in \eqref{eq:wp1} belongs to the kernel of the operator $\Acal(\omega,\delta)$, namely the condition \eqref{eq:conre} is fulfilled. In this case, condition \eqref{eq:rbnu} signifies that $\bnu$ should be an eigenfunction of the operator $\bK_{\partial D}^{\omega,*}$. In fact, this is the case when the domain $D$ is a ball. In \cite{DLL1}, \mm{it was proved} that $\bnu$ is an eigenfunction of the operator $\bK_{\partial D}^{\omega,*}$, namely
\begin{equation}\label{eq:eiK}
  \bK_{\partial D}^{\omega,*}[\bnu] = \chi_1 \bnu, \quad \bx\in \partial D,
\end{equation}
where 
\[
 \chi_1= \frac{4\rmi\mu R k_p }{(\lambda+2\mu)} j_{1}(k_p R) h_{1}(k_p R) - \rmi R^2 k_p^2  j_{1}(k_p R) h_{0}(k_p R)-\frac{1}{2},
\]
with $R$ being the radius of the ball $D$, and $j_n(|\bx|)$ and $h_n(|\bx|)$ respectively denoting the spherical Bessel function and spherical Hankel function of the first kind and of order $n$. Moreover, $\bnu$ is also an eigenfunction of the operator $ \bS_{\partial D}^{\omega}$, namely
\begin{equation}
\begin{split}
 \bS_{\partial D}^{\omega}[ \bnu](\bx) =\frac{-\rmi R^2 k_p}{(\lambda+2\mu)}   h_{1}(k_p R) j_{1}(k_p R) \bnu, \quad \bx\in \partial D.
\end{split}
\end{equation}
\mm{It was also proved} in \cite{s25} that
\begin{equation}
   S_{\partial D}^{k}[1](x)=-\rmi k R^2 h_0(k R) j_0(k R), \quad \bx\in \partial D,
\end{equation}
and 
\begin{equation}\label{eq:eik}
  K_{\partial D}^{k,*}[1](x)= \frac{1}{2}-\rmi k^2 R^2 j_0^{\prime}(kR) h_0(kR), \quad \bx\in \partial D.
\end{equation}
Following the asymptotic expansions for the functions $j_n(|\bx|)$ and $h_n(|\bx|)$, $n=0,1$, with $|\bx|\ll1$ (cf. \cite{CK:bk1}), one can obtain the expressions of $S_{\partial D,i}[1]$, $\mathbf{R}_{1,i}[\bnu]$ and $\mathbf{R}_{2,i}[\bnu]$ for $i\geq0$, respectively. Next we only present the first few terms,
\[
 \delta \tau^2   +\frac{4}{3} \mu-\rmi k \left( \frac{3\tau^3 \delta + 4 \tau \mu}{3\sqrt{\lambda+2\mu}} \right)  - k^2\left( \frac{1}{3} \tau^2 + \frac{1}{6} \tau^2\delta \right)  + k^3\frac{\rmi \tau^3\delta}{6\sqrt{\lambda+2\mu}} +\cdots.
\]
One can readily see the equation \eqref{eq:br2} is reduced to solving a polynomial equation with respect to $k$ of an infinity order. By a truncation and approximation, we solve the following equation,
\begin{equation}\label{eq:3dsp}
  \delta \tau^2   +\frac{4}{3} \mu-\rmi k \left( \frac{3\tau^3 \delta + 4 \tau \mu}{3\sqrt{\lambda+2\mu}} \right)  - k^2\left( \frac{1}{3} \tau^2 + \frac{1}{6} \tau^2\delta \right)=0.
\end{equation}
whose roots are given by 
\begin{equation}\label{eq:root3}
k_{d3\pm}=\frac{\pm\sqrt{(3 \tau^2 \delta +4\mu)(4(\lambda+\mu)-3\tau^2\delta)}-(3\tau^2\delta + 4\mu)\rmi }{2\tau\sqrt{\lambda+2\mu}}.
\end{equation}
One can verify directly that the root \eqref{eq:re3} is the positive part of the roots \eqref{eq:root3} neglecting the infinitesimal part. 
In fact, the critical values obtained in \eqref{eq:root3} exhibit excellent accuracy for the resonant frequencies; see Remarks \ref{rem:ex3} and \ref{rem:num} \mm{in what follows.}
\end{rem}

\section{Minnaert resonances in two dimensions}

In this section, we derive the Minnaert resonances for the system \eqref{eq:mond} in two dimensions when the domain $D$ is a unit disk. \rr{The extension of the low-frequency analysis from three dimensions to two dimensions is technically not straightforward. A major difficulty comes from the fact that the asymptotic expansions of the fundamental solutions $G^k(\bx)$ in 2D and 3D defined in \eqref{eq:fu_he} are of a different nature. In fact, the expansion of the fundamental solution $G^k(\bx)$ in 3D is the summation of $\varphi_j(\bx)k^j$ with $j=0,1,\cdots$. However in 2D the asymptotic expansion is the summation of $\phi(\bx)(c_j+\ln(k))k^j$, with $j=0,1,\cdots$(cf.\cite{AF1}), which significantly increases the complexity of solving the counterpart equation \eqref{eq:conre} in the 2D case. Hence, for this technical reason, we shall only derive the Minnaert resonances for the system \eqref{eq:mond} in two dimensions for a disk domain $D$. Indeed, as can be seen from Theorem \ref{thm:wr2d}, even the domain $D$ is a disk in 2D, one can not derive the explicit expression of the resonant frequency.}


In what follows, we let $J_{n}(|\bx|)$ and $H_{n}(|\bx|)$ respectively denote the Bessel function of order $n$ and the Hankel function of the first kind of order $n$. When the argument $k\ll1$, the functions $J_n$ and $H_{n}$, $n=0,1$, enjoy the following asymptotic expansions (cf.\cite{CK}):
\begin{equation}\label{eq:asJ}
 J_0(k) = 1- \frac{k^2}{4}+\frac{k^4}{64} +\mathcal{O}(k^6), \qquad  J_1(k) = \frac{k}{2}-\frac{k^3}{16} + \mathcal{O}(k^5),
\end{equation}
\begin{equation}\label{eq:asH0}
 H_0(k) =  \frac{\rmi ( \gamma + 2\ln (k))}{\pi}+\frac{\rmi ( -2+\gamma + 2\ln (k))k^2}{4\pi} + \mathcal{O}((1+\ln(k))k^3),
\end{equation}
and 
\begin{equation}\label{eq:asH1}
 H_1(k) =  -\frac{2\rmi}{k\pi} + \frac{\rmi (-1+ \gamma + 2\ln (k))k}{2\pi} + \mathcal{O}((1+\ln(k))k^3),
\end{equation}
\mm{with
$
 \gamma=2E_c-\rmi \pi -2 \ln 2,
$
and $E_c$ being the Euler's constant.
}

 By the definition of the strong resonance in \eqref{eq:conre} for the system \eqref{eq:mond} , we next construct a nontrivial solution $\Phi$ such that
\begin{equation}\label{eq:re2d}
 \Acal(k,\delta)[\Phi](\bx)=0,
\end{equation}
where $\Acal(k,\delta)$ is defined in \eqref{eq:defA}. If the domain $D$ is a unit disk, direction calculations show that for $\bx\in \partial D$,
\begin{equation}\label{eq:eil2d}
 {\bS}_{\partial D}^{k\tau}[\bnu](\bx) =\zeta_1 \bnu \quad \mbox{and}\quad {\bK}^{k\tau,*}_{\partial D}[\bnu](\bx) =  \zeta_2 \bnu, 
\end{equation}
where
\[
 \zeta_1=\frac{-\rmi\pi}{2(\lambda+ 2\mu)} J_1(k_p) H_1(k_p),
\]
and 
\[
 \zeta_2= \frac{-\rmi \pi J_1(k_p)}{2(\lambda+ 2\mu)}\left( (\lambda+ 2\mu)k_p H_1^{\prime}(k_p) + \lambda H_1(k_p)  \right)-\frac{1}{2},
\]
with  $k_p=k\tau/\sqrt{\lambda+2\mu}$.
Moreover, one has that for $\bx\in \partial D$ (cf.\cite{FDC}), 
\begin{equation}\label{eq:eih2d}
  S^k_{\partial D}[1](\bx) = \zeta_3  \quad \mbox{and}\quad K^{k,*}_{\partial D}[1](\bx) =\zeta_4,
\end{equation}
where
\[
   \zeta_3 = \frac{-\rmi\pi}{2} J_0(k) H_0(k) \quad \mbox{and}\quad \zeta_4= \frac{1}{2}- \frac{\rmi\pi}{2} k J_0^{\prime}(k) H_0(k).
\]
Hence the nontrivial solution to the equation \eqref{eq:re2d} should have the following form 
\[
 \Phi= \left(
    \begin{array}{c}
      b_1 \\
      b_2\bnu  \\
    \end{array}
  \right) .
\]
Substituting the last equation into \eqref{eq:re2d} yields that
\begin{equation}\label{eq:red2d}
\mathbf{B}\mathbf{b}=0,
\end{equation}
where
\[
  \mathbf{B}=\left(
    \begin{array}{cc}
        \frac{1}{k^2}\left(-\frac{1}{2} + \zeta_4\right) &  -\zeta_1\medskip \\
      \delta \tau^2\zeta_3 & \frac{1}{2} + \zeta_2\\
    \end{array}
  \right) 
\quad \mbox{and}\quad
\mathbf{b}=
\left(
    \begin{array}{c}
      b_1 \\
      b_2  \\
    \end{array}
  \right) 
\]
with $\zeta_i$, $i=1,2,3,4$ defined in \eqref{eq:eil2d} and \eqref{eq:eih2d}. To ensure that the equation \eqref{eq:red2d} possesses nontrivial solutions, the determinant $\det(\mathbf{B})$ of the matrix $\mathbf{B}$ should vanish. 
Through some straightforward but rather tedious calculations and with the help of the asymptotic expressions in \eqref{eq:asJ}, \eqref{eq:asH0} and \eqref{eq:asH1}, we can obtain
\[
\begin{split}
 \det(\mathbf{B})  &=  \frac{1}{k^2}\left(-\frac{1}{2} + \zeta_4\right)\left( \frac{1}{2} + \zeta_2\right) +\delta \tau^2  \zeta_1 \zeta_3\\
&= -(\gamma + 2 \ln(k)) \left(\frac{(\mu+\delta\tau^2)}{4(\lambda + 2\mu)} + \frac{k^2 \tau^2\lambda(\gamma+ 2\ln(k\tau/\sqrt{\lambda+2\mu}))}{16(\lambda + 2\mu)^2} \right)+ \\
&\quad o(\mu(\gamma + \ln(k))) + o(\delta(\gamma + \ln(k))) +  o(k^2(\gamma + \ln(k))),
\end{split}
\]
where $\gamma$ is defined in \eqref{eq:asH0}. Hence, \mm{we readily come to the following conclusion.}
\begin{thm}\label{thm:wr2d}
Consider the system \eqref{eq:mond} in two dimensions with $D$ being a central disk. If the parameters are chosen according to \eqref{eq:paas}, then the strong resonance occurs. 
Moreover, the leading-order terms of the resonant frequencies are given by the roots of the following equation
\begin{equation}\label{eq:ef2d}
 (\gamma + 2 \ln(k)) \left(\frac{(\mu+\delta\tau^2)}{4(\lambda + 2\mu)} + \frac{k^2 \tau^2\lambda(\gamma+ 2\ln(k\tau/\sqrt{\lambda+2\mu}))}{16(\lambda + 2\mu)^2} \right)=0,
\end{equation}
where $\tau$ and $\gamma$ are given in \eqref{eq:dtau} and \eqref{eq:asH0}, respectively.
\end{thm}

\begin{rem}\label{rem:ex3}
The method used above in deriving the resonances in two dimensions can be applied to the three dimensions as well when the domain $D$ is a central ball. From \eqref{eq:eiK} to \eqref{eq:eik}, one can calculate in a similar manner the determinant of the matrix $\mathbf{B}$ in three dimensions and determine the critical values $k$'s such that the following 
\[
 \det(\mathbf{B}) =0,
\]
 holds to ensure the occurrence of the strong resonance.
\end{rem}
\begin{rem}\label{rem:num}
There exist critical values $k$'s such that $\det(\mathbf{B}) $ vanishes in both two and three dimensions, that is, 
the strong resonance occurs. Since the expression of $\det(\mathbf{B}) $ is nonlinear with respect to $k$, we can resort to computational algorithms to determine these critical values, namely resonant frequencies. Next, for illustrations, we conduct some numerical experiments to find out these critical values. We denote by $k_{b2}$ and $k_{b3}$ for the critical values by directly solving the equation $ \det(\mathbf{B}) =0$ in two and three dimensions, respectively. As comparisons, we also calculate $k_{d3+}$ defined in \eqref{eq:root3} and solve the equation \eqref{eq:ef2d}. The root of the equation \eqref{eq:ef2d} is denoted by $k_{d2}$. The parameters in our numerical experiments are chosen as follows:
\[
 \lambda=1,\quad \tau=1, \quad \mu=\delta=10^{-i} , \  i=2,3,4.
\]
Moreover, the bubble $D$ is a unit disk in $\mathbb{R}^2$ and a unit ball in $\mathbb{R}^3$. 
It is remarked that the case $i=3$ almost indicates the experiment in \cite{CTL}. The corresponding values, $k_{b2}, k_{d2}, k_{b3}$ and $k_{d3+}$ with positive real parts, are presented in Table \ref{tab6}. From Table \ref{tab6}, one can conclude that there indeed exist critical values $k$ such that $ \det(\mathbf{B}) =0$ in both two and three dimensions. Moreover, the roots of the equations \eqref{eq:3dsp} and \eqref{eq:ef2d} exhibit an excellent accuracy agreement with the resonant frequencies. Finally, we would like to point out the negative imaginary parts in the values computed in Table 1 are a physically reasonable requirement (cf. \cite{AF1}). 
\begin{table}[t]
    \centering
    \begin{tabular}{cccc}
      \toprule
      & $i=2$ & $i=3$ & $i=4$ \\
      \midrule
      $k_{b2}$ & $0.110087 - 0.040732\rmi$  & $0.030796 - 0.007347\rmi$ &  $0.008681 - 0.001513\rmi$ \\[5pt]
      $k_{d2}$ & $0.109963 - 0.040294\rmi$  & $0.030790 - 0.007341\rmi$ &  $0.008681 - 0.001513\rmi$ \\[5pt]
      $k_{b3}$ & $0.262065 - 0.034521\rmi$  & $0.083584 - 0.003495\rmi$ & $0.026454 - 0.000349\rmi$  \\[5pt]
      $k_{d3+}$ & $0.262296 - 0.034655\rmi$  & $0.083592 - 0.003496\rmi$  & $0.026455 - 0.000349\rmi$  \\
      \bottomrule
    \end{tabular}
  \caption{ The critical values of  $k_{b2}, k_{d2}, k_{b3}$ and $k_{d3+}$ with positive real parts.}
  \label{tab6}
\end{table}
\end{rem}

\section{Concluding remarks}

\mm{
We have studied the Minnaert resonances for bubble-elastic structures. 
By delicately and subtly balancing the acoustic and elastic parameters as well as the geometry of the bubble, 
we have shown that the Minnaert resonance can (at least approximately) occur for rather general constructions. 
Our study opens up a new direction for the mathematical investigation on bubbly elastic mediums with many potential developments. In the present paper, we have considered only the case that a single bubble is embedded in a soft elastic material. It would be interesting to consider the case with multiple bubbles as well as the corresponding application to the effective realisation of elastic metamaterials. Moreover, we have investigated only the case that the resonance is mainly caused by the p-wave, but it would be interesting to investigate more general bubbly elastic structures with more general resonances. We shall consider these and other related topics in our forthcoming work. 
}

\section*{Acknowledgement}

 The work of H Liu was supported by the FRG fund from Hong Kong Baptist University and the Hong Kong RGC General Research Fund (projects 12302919, 12301218 and 12302017).  The work of J Zou was supported by the Hong Kong RGC General Research Fund (project 14304517) and 
NSFC/Hong Kong RGC Joint Research Scheme 2016/17 (project N\_CUHK437/16).

\end{document}